\documentclass[11pt,onecolumn,final,reqno,fleqn]{article}

\usepackage[margin=1.2in]{geometry}
\usepackage{graphicx}
\usepackage{subcaption}  % needed for sub figures
\usepackage{enumerate}
\usepackage{paralist}
\usepackage{xcolor} 
\usepackage[hypcap]{caption}
\usepackage[colorlinks=true,citecolor=cyan]{hyperref}		% comment in final draft
\usepackage[sort&compress,numbers]{natbib}

\usepackage{tikz-cd}

\usepackage{amsmath}
\usepackage{amsthm}
\usepackage{amsfonts}
\usepackage{amssymb}
\usepackage{mathrsfs}  % nice script fonts
\usepackage{mathtools}
\usepackage{bm} % bold math
\usepackage{bbm} % blackboard math \mathbbm

\newcommand{\mc}{\mathcal}

\newcommand{\mbf}{\bm} 			% requires amsmath, amssymb packages
\newcommand{\N}{ \ensuremath{\mathbb{N}}}  % natural numbers
\newcommand{\R}{ \ensuremath{\mathbb{R}}}  % real numbers
\newcommand{\C}{ \ensuremath{\mathbb{C}}}  % complex numbers
\newcommand{\eps}{\epsilon}

\newcommand{\lam}{\lambda}

\providecommand\given{} % make sure it exists
\newcommand\SetSymbol[1][]{
   \nonscript\,#1\vert \allowbreak \nonscript\,\mathopen{}}

\DeclarePairedDelimiterX\set[1]{\lbrace}{\rbrace}{ \renewcommand\given{\SetSymbol[\delimsize]} #1 }  % set* autoscales
\DeclarePairedDelimiterX\norm[1]{\lVert}{\rVert}{#1}  			% norm (\norm*{} autoscales or \norm[\big]{})
\DeclarePairedDelimiterX\inner[2]{\langle}{\rangle}{#1 \,, #2}  	% inner product
\DeclareMathOperator{\diag}{diag}								% Diagonal matrix.
\DeclarePairedDelimiterX\abs[1]{\lvert}{\rvert}{#1} % absolute value: \abs{} = no-resize, \abs*{} = left/right auto-resize, \abs[size-cmd]{} = size manually adjusted with size-cmd = \big,\Big,\bigg,\Bigg

\theoremstyle{plain}
  \newtheorem{theorem}{Theorem}[section]
  \newtheorem{lemma}[theorem]{Lemma}
  \newtheorem{proposition}[theorem]{Proposition}

  \newtheorem*{corollary}{Corollary}
\theoremstyle{remark}
  \newtheorem*{remark}{Remark}

\theoremstyle{definition}
  
  \newtheorem{definition}[theorem]{Definition}

\usepackage[T1]{fontenc}
\begin{document}
\title{Koopman principle eigenfunctions and linearization of diffeomorphisms}
%\subtitle{}
\author{Ryan Mohr, Igor Mezi\'{c}}
%\publishers{}
%\subject{}
\date{}

\maketitle

% Main Body
\begin{abstract}
This paper considers a nonlinear dynamical system on a complex, finite dimensional Banach space which has an asymptotically stable, hyperbolic fixed point.  We investigate the connection between the so-called principle eigenfunctions of the Koopman operator and the existence of a topological conjugacy between the nonlinear dynamics and its linearization in the neighborhood of the fixed point. The principle eigenfunctions generate an algebra of observables for the linear dynamics --- called the principle algebra --- which can be used to generate a sequence of approximate conjugacy maps in the same manner as it is done in normal form theory. Each element of the principle algebra has an expansion into eigenfunctions of the Koopman operator and composing an eigenfunction with one of the approximate topological conjugacies gives an approximate eigenfunction of the Koopman operator associated with the nonlinear dynamical system. When the limit of the approximate conjugacies exists and attention is restricted to real Banach spaces, a simple application of the Stone-Weierstrass theorem shows that both the principle algebra and the pull-back algebra --- defined by composing the principle algebra with the topological conjugacy --- are uniformly dense in either the space of continuous functions or the maximal ideal of continuous functions vanishing at the fixed point. The point is that, a priori, it is difficult to know which space of observables to use for dissipative nonlinear dynamical systems whose elements have spectral expansions into eigenfunctions. These results say that any continuous observable is arbitrarily close to one that has such an expansion.
\end{abstract}

%\tableofcontents

%%%%%%%%%%%%%%%%%%%%%%%%%%%%%%%%%%%%%%%%%%%%%%%%%%%%%%%%%%%%%%%%%%%%%%%%%%%%%%%%%%%%%%%%%%%%%%%%
\section{Introduction}
In this paper, we are interested in the connections between the principle eigenfunctions of the Koopman operator associated with a diffeomorphism of a Banach space with an asymptotically stable, hyperbolic fixed point and the existence of a topological conjugacy between the diffeomorphism and its linearization around the fixed point. As long as the principle eigenvalues satisfy the standard non-resonance conditions (see \cite{Arnold:1989wt,Wiggins:2003tz,Katok:1997th}) of normal form theory, the principle eigenfunctions generate a commutative algebra of polynomials from which a sequence of coordinate transformations that linearizes the diffeomorphism can be generated. More specifically, if $\mc A_{T'(0)}$ is the generated algebra and $T : X \to X$ is the diffeomorphism, then there are Banach-valued polynomials $\tau_{k} : X \to X$ in $X \otimes \mc A_{T'(0)}$ such that 
	%% ------------ EQUATION ------------ %%
	\begin{equation}
	T_{k}(x) := (\tau_{k}^{-1} \circ T \circ \tau_{k})(x) = T'(x_0)(x - x_0) + R_{k+1}(x), \qquad (x\in X)
	\end{equation}
	%% ---------------- END  ---------------- %%
where $x_0$ is the fixed point and $R_{k+1}$ is a power series around $x_0$ such that $R_{k+1}(x) \in \mc{O}(\norm{x - x_0}^{k+1})$ as $x \to x_0$.

Our interest in this connection arose when investigating the construction of spaces of observables for nonlinear dissipative dynamics, each element of which has an expansion into eigenfunctions of the associated Koopman operator. The procedure for constructing such spaces is straight forward and proceeds by first defining a space of observables for the simpler linear dynamics. Let $U_{T}$ and $U_{T'(x_0)}$ be the Koopman operators associated with $T$ and $T'(x_0)$, respectively, and assume that we have constructed a space of observables $\mc F_{T'(x_0)}$ for $U_{T'(x_0)}$ of which every element has an expansion into eigenfunctions. If there exists a topological conjugacy $\tau : X \to X$ ($T'(x_0) = \tau^{-1} \circ T \circ \tau$), then every element of the space $\mc F_{T} := \mc F_{T'(x_0)} \circ \tau^{-1} = \set{f\circ \tau^{-1} \given f \in \mc F_{T'(x_0)}}$ has an expansion into eigenfunctions of $U_{T}$. This follows easily from the fact that if $\phi$ is an eigenfunction of $U_{T'(x_0)}$, then $\phi \circ \tau^{-1}$ is an eigenfunction of $U_{T}$ \cite{Budisic:2012cf}. However, getting a concrete representation of elements of $\mc F_{T}$ in terms of simpler functions can be difficult due to the action of $\tau^{-1}$ --- we know that $\mc F_{T}$ and $\mc F_{T'(x_0)}$ are isomorphic through $\tau$, but we do not know if they are \emph{identical}. If we can show that the two spaces are identical, or at least their uniform closures agree, then we can use the functions in $\mc F_{T'(x_0)}$ and its uniform closure as observables for the nonlinear dynamics and still have spectral expansions for each element.

In this paper, we answer the following questions.
\begin{enumerate}[(1)]
\item\label{question:1} Is there a sub-algebra of $C(X)$ of which every element has an expansion into eigenfunctions of $U_{T'(x_0)}$?
\item\label{question:2} Similarly, is there a sub-algebra of $C(X)$ of which every element has an expansion into eigenfunctions of $U_{T}$?
\item\label{question:3} If the above two sub-algebras exist, what is their relation?
\end{enumerate}

Answering \eqref{question:1}, the so-called principle eigenfunction can be used to generate an algebra of observables (called the principle algebra) which have expansions in eigenfunctions. This algebra can be pulled back to an algebra of observables for the nonlinear dynamics as long as a topological conjugacy exists. This conjugacy can actually be constructed from change of variables transformations that are formed from tensor products of the principle algebra with the Banach space (question \eqref{question:3}). The pullback algebra, formed via the topological conjugacy, also has spectral expansions since eigenfunctions are preserved when composing with topological conjugacies \cite{Budisic:2012cf} (question \eqref{question:2}).

If we work on real, finite dimensional, Banach spaces and look at the complexification of the linearization, the generated principle algebra separates points and is closed under complex conjugation. If we restrict the domain of the algebra to a compact neighborhood of the fixed point and appeal to the Stone-Weierstrass theorem, the uniform closure of the principle algebra is either the space of complex continuous functions on the compact neighborhood (if the constant functions are appended to the algebra) or the maximal ideal consisting of complex continuous functions which vanish at the fixed point. Furthermore, if we have built the topological conjugacy $\tau$ from the linear to nonlinear dynamics, we can pull back the principle algebra of observables for the linear dynamics to an algebra of observables for the nonlinear dynamics by composing functions $f$ in the principle algebra with $\tau^{-1}$; i.e., elements of the algebra of observables for the nonlinear dynamics have the form $f \circ \tau^{-1}$. Taking a compact neighborhood $K$ of the fixed point, the sets $L = \tau(K)$ and $K$ are homeomorphic via the conjugacy. This implies that the pull back algebra of observables for the nonlinear dynamics inherits the uniform closure properties of the principle algebra of observables; i.e., if the uniform closure of the principle algebra of observables is the entire space of complex continuous functions on $K$, then the uniform closure of the pull back algebra is the space of complex continuous functions on $L = \tau(K)$. If instead the uniform closure of the principle algebra of observables is the maximal ideal $I_{x_{0}}(K) = \set{ f \in C_{\C}(K) \given f(x_0) = 0}$, then the uniform closure of the pull back algebra is the maximal ideal $I_{x_{0}}(L) = \set{ g \in C_{\C}(L) \given g(x_0) = 0 }$, where $x_0$ is the fixed point.

One can interpret these results as saying that every continuous function is a perturbation of an observable that has an expansion into eigenfunctions of the associated Koopman operator. Or, if we choose a continuous function as an observable, then there is another observable arbitrarily close in the uniform norm that has a spectral expansion into Koopman eigenfunctions.

The rest of the paper is structured as follows. In section \ref{sec:preliminaries} we develop the framework for polynomial transformations between Banach spaces and the concept of principle eigenfunctions of the Koopman operator. The definition of the polynomial maps requires the introduction multilinear maps between Banach spaces. We record these definitions and basic properties in section \ref{subsec:polynomials}. In section \ref{subsec:principle-eigenfunctions}, we define the relevant Koopman operators, (Koopman) principle eigenfunctions, the non-resonance condition on the principle eigenvalues, and the principle algebra. Section \ref{sec:linear-nonlinear-observables} relates the principle algebra of observables for the linearized system with an algebra of observables for the nonlinear system. This is done by constructing a sequence of approximate topological conjugacies between the linear and nonlinear dynamics using the principle algebra. The results of this section follow normal form theory closely. Section \ref{sec:approximate-conjugacies-eigenfunctions} relates the spectral properties of the Koopman operator for the linearized system with those of the Koopman operator for the nonlinear system. It is shown that if an eigenfunction for the linearized system's Koopman operator is composed with the sequence of approximate topological conjugacies, then the resulting functions are approximate eigenfunctions of the Koopman operator of the nonlinear system at the same eigenvalue when restricted to the open unit ball. By approximate eigenfunction, we mean a function that is order $\epsilon$ away from a true eigenfunction (see \eqref{eq:approximate-eigenfunction}). In section \ref{sec:uniform-closure-of-A}, we restrict our Banach space to be real. After going through the standard complexification process of the space and the map $T$, an application of the Stone-Weierstrass theorem gives that both the principle algebra of observables for the linear system and its pullback to the nonlinear system is either uniformly dense in the continuous functions or the maximal ideal of continuous functions that vanish at the fixed point. The paper is concluded in section \ref{sec:conclusions}.

%%%%%%%%%%%%%%%%%%%%%%%%%%%%%%%%%%%%%%%%%%%%%%%%%%%%%%%%%%%%%%%%%%%%%%%%%%%%%%%%%%%%%%%%%%%%%%%%
\section{Preliminaries}\label{sec:preliminaries}
Let $X$ be a finite-dimensional, complex Banach space of dimension $n$ with norm $\norm{\cdot}$. Denote by $\inner{\cdot}{\cdot} : X \times X^* \to \C$ the bilinear form defined as $\inner{x}{f} = f(x)$.

\subsection{Polynomials on Banach spaces}\label{subsec:polynomials}
In order to define Banach-valued polynomials on $X$, we need to define multilinear maps between Banach spaces; the polynomials are the diagonal of a multilinear map. Let $Y$ be a Banach space and $A : \prescript{m}{}{X} \to Y$ be an $m$-linear mapping; that is $A(x_1,\dots, x_m)$ is linear in each variable $x_i \in X$. The norm of an $m$-linear map is defined as
	%% ------------ EQUATION ------------ %%
	\begin{equation}
	\norm{A} = \sup\set{ \norm{A(x_1,\dots,x_m)} \given x_i \in X, \norm{x_i} \leq 1}.
	\end{equation}
	%% ---------------- END  ---------------- %%
$A$ is called an $m$-linear form if $Y = \C$.
An $m$-linear map can be constructed from elements of the form $a(x_1,\dots, x_m) =  (\xi_{1}(x_1)\cdots \xi_{m}(x_{k}))\cdot y$, where $y \in Y$ and $\xi_{1},\dots, \xi_{m} \in X^{*}$. Note that there is a representation \cite[see][sec.\ 1.5]{Ryan:2002ku} of $\mc L(\prescript{m}{}{X}; Y)$ as a tensor product of the Banach space $Y$ with the space of $m$-linear forms on $X^m$:
	%% ------------ EQUATION ------------ %%
	\begin{equation}\label{eq:vector-valued-m-linear-function}
	\mc L(\prescript{m}{}{X}; Y) = \mc L(\prescript{m}{}{X}; \C) \otimes Y.
	\end{equation}
	%% ---------------- END  ---------------- %%
For $m=0$, $\mc L(\prescript{m}{}{X}; Y) \cong Y$. If $f \in \mc L(\prescript{m}{}{X}; \C)$, then $f \otimes y  \in  \mc L(\prescript{m}{}{X}; \C)  \otimes Y$ is the map from $X$ to $Y$ defined as $(f \otimes y)(x) :=  f(x) y$.

Since $X$ is finite dimensional, there is a canonical representation for $m$-linear maps. Let $\set{e_1, \dots, e_n}$ be a basis for $X$ with coordinate functionals $\set{\phi_1, \dots, \phi_n } \subset X^*$ ($\inner{e_i}{\phi_j} = \delta_{i,j}$). Each $m$-linear map can be uniquely represented as
	%% ------------ EQUATION ------------ %%
	\begin{equation}\label{eq:canonical-m-linear-map}
	A(x_1, \dots , x_m) = \sum_{j_1 = 1}^{n} \cdots \sum_{j_m = 1}^{n}  y_{j_1,\dots, j_m} \phi_{j_1}(x_1)\cdots \phi_{j_m}(x_m) \qquad (x_1,\dots, x_m \in X),
	\end{equation}
	%% ---------------- END  ---------------- %%
where $ y_{j_1,\dots, j_m} \in Y$. 

\begin{proof}[Proof of \eqref{eq:canonical-m-linear-map}]
Let $a_{i} = y_{i} \cdot (\xi_{i,1}\cdots \xi_{i,m}) \in \mc L(\prescript{m}{}{X}; Y)$, where $y_i \in Y$ and $\xi_{i,j} \in X^{*}$, and define $A \in \mc L(\prescript{m}{}{X}; Y)$ as $A = \sum_{i=1}^{s} a_{i}$. Let $\set{ \phi_{1}, \dots, \phi_{n} }$ be any basis for $X^{*}$. Then each $\xi_{i,j}$ has a unique expression in terms of the $\phi_{k}$'s as
	%% ------------ EQUATION ------------ %%
	\begin{equation*}
	\xi_{i,j} = \sum_{k_{j}=1}^{n} c_{i,j,k_j} \phi_{k_{j}}, \qquad (c_{i,j,k_j} \in \C).
	\end{equation*}
	%% ---------------- END  ---------------- %%
Then $a_{i}$ becomes
	%% ------------ EQUATION ------------ %%	
	\begin{align*}
	a_{i} = y_{i} \cdot (\xi_{i,1}\cdots \xi_{i,m}) 
	&= y_{i} \Big( \sum_{k_{1}=1}^{n} c_{i,1,k_1} \phi_{k_{1}} \Big) \cdots \Big( \sum_{k_{m}=1}^{n} c_{i,m,k_m} \phi_{k_{m}} \Big) \\
	&= \sum_{k_{1}=1}^{n} \cdots \sum_{k_{m}=1}^{n}  y_{i} \Big(c_{i,1,k_1}\cdots c_{i,m,k_m} \Big) \Big(\phi_{k_{1}}  \cdots \phi_{k_{m}} \Big) \\
	&= \sum_{k_{1}=1}^{n} \cdots \sum_{k_{m}=1}^{n}  y_{i,k_{1},\dots, k_{m}} \Big(\phi_{k_{1}}  \cdots \phi_{k_{m}} \Big)
	\end{align*}
	%% ---------------- END  ---------------- %%
where $y_{i,k_{1},\dots,k_{m}} \in Y$ was defined as $y_{i,k_{1},\dots,k_{m}} = y_{i}\cdot  \Big(c_{i,1,k_1}\cdots c_{i,m,k_m} \Big)$. Finally,
	%% ------------ EQUATION ------------ %%
	\begin{align*}
	A = \sum_{i=1}^{s} a_{i} &= \sum_{i=1}^{s} \sum_{k_{1}=1}^{n} \cdots \sum_{k_{m}=1}^{n}  y_{i,k_{1},\dots, k_{m}} \Big(\phi_{k_{1}}  \cdots \phi_{k_{m}} \Big) \\
	&= \sum_{k_{1}=1}^{n} \cdots \sum_{k_{m}=1}^{n}  y_{k_{1},\dots, k_{m}} \Big(\phi_{k_{1}}  \cdots \phi_{k_{m}} \Big)
	\end{align*}
	%% ---------------- END  ---------------- %%
where $y_{k_{1},\dots, k_{m}} := \sum_{i=1}^{s} y_{i,k_{1},\dots,k_{m}}$. This completes the proof of the canonical form \eqref{eq:canonical-m-linear-map}.
\end{proof}

If $A$ is an $m$-linear map, $x_1,\dots,x_n \in X$, and $\alpha_1,\dots, \alpha_n \in \N_0$, with $\sum_{1}^{n} \alpha_i = m$, we write 
	%% ------------ EQUATION ------------ %%
	\begin{equation*}
	A(x_1^{\alpha_1},\dots, x_{n}^{\alpha_n}) = A(\underbrace{x_1,\dots, x_{1}}_{\alpha_1\text{ times}},\dots, \underbrace{x_n,\dots, x_n}_{\alpha_n\text{ times}}) .
	\end{equation*}
	%% ---------------- END  ---------------- %%
$A$ is a symmetric $m$-linear form if $A(x_1,\dots,x_m) = A(x_{\sigma(1)},\dots, x_{\sigma(m)})$ for all $\sigma \in S_m$, where $S_{m}$ is the group of permutations on $m$ elements. To each $m$-linear map $A$, there is an associated symmetric $m$-linear map that can be constructed as
	%% ------------ EQUATION ------------ %%
	\begin{equation}\label{eq:symmetrized-m-linear-form}
	A^s(x_{1},\dots,x_{m}) = \frac{1}{m!} \sum_{\sigma \in S_{m}} A(x_{\sigma(1)},\dots, x_{\sigma(m)}).
	\end{equation}
	%% ---------------- END  ---------------- %%
Newton's binomial formula holds for symmetric $m$-linear forms:
	%% ------------ EQUATION ------------ %%
	\begin{equation}\label{eq:binomial-formula}
	A^{s}((x+y)^{m}) = \sum_{j=0}^{m} \binom{m}{j} A^{s}(x^{m-j},y^j).
	\end{equation}
	%% ---------------- END  ---------------- %%

A mapping $P : X \to Y$ is said to be an $m$-homogeneous polynomial if there is an $m$-linear mapping $A : \prescript{m}{}{X} \to Y$ such that
	%% ------------ EQUATION ------------ %%
	\begin{equation}\label{eq:polynomial-m-linear-map}
	P(x) = A(x,\dots,x) \qquad (x \in X).
	\end{equation}
	%% ---------------- END  ---------------- %%
It is clear that $P(x) = A(x,\dots, x) = A^{s}(x,\dots,x)$.
The norm of $P$ is defined as the quantity
	%% ------------ EQUATION ------------ %%
	\begin{equation}
	\norm{P} = \sup \set{ \norm{P(x)} \given x\in X, \norm{x}\leq 1},
	\end{equation}
	%% ---------------- END  ---------------- %%
and for arbitrary $x\in X$, $\norm{P(x)} \leq \norm{P} \norm{x}^{m}$. Let $\mc P_a(\prescript{m}{}{X}; Y)$ denote the vector space of all $m$-homogeneous polynomials from $X$ into $Y$ and $\mc P(\prescript{m}{}{X}; Y)$ the subspace of all continuous polynomials.\footnote{These are not equal in general if $X$ is infinite-dimensional. However, in our case they are equivalent.} $\mc P_a(X; Y)$ will denote the algebraic direct sum of the vector spaces $\mc P_a(\prescript{m}{}{X}; Y)$ with $m \in \N_0$, while $\mc P(X; Y)$ is the algebraic direct sum of the subspaces $\mc P(\prescript{m}{}{X}; Y)$ for $m\in\N_0$.  By \eqref{eq:polynomial-m-linear-map} and \eqref{eq:vector-valued-m-linear-function}, $\mc P_a(\prescript{m}{}{X}; Y)$ has the representation 
	%% ------------ EQUATION ------------ %%
	\begin{equation}\label{eq:vector-valued-m-polynomial-tensor-product}
	\mc P_a(\prescript{m}{}{X}; Y) =  \mc P_a(\prescript{m}{}{X}; \C) \otimes Y.
	\end{equation}
	%% ---------------- END  ---------------- %%
It follows from \eqref{eq:canonical-m-linear-map} that every $P \in \mc P_a(\prescript{m}{}{X};Y)$ can be uniquely represented as
	%% ------------ EQUATION ------------ %%
	\begin{equation}\label{eq:canonical-polynomial-map}
	P(x) = \sum_{\abs{\mbf \alpha} = m} y_{\mbf \alpha} (\phi_1^{\alpha_1}(x) \cdots \phi_n^{\alpha_n}(x))
	\end{equation}
	%% ---------------- END  ---------------- %%
where $\alpha = (\alpha_1,\dots, \alpha_n) \in \N_0^n$, $\abs{\alpha} = \sum_{i=1}^{n} \abs{\alpha_i}$, and $y_{\alpha} \in Y$. Furthermore, if $Y = X$, then $P$ has a representation
	%% ------------ EQUATION ------------ %%
	\begin{equation}
	P(x) = \sum_{i=1}^{n} \sum_{\abs{\alpha} = m} x_{i,\alpha} (\phi_1^{\alpha_1}(x)\cdots \phi_n^{\alpha_n}(x)) e_i
	\end{equation}
	%% ---------------- END  ---------------- %%
where $x_{i,\alpha} \in \C$. $\mc P_a(\prescript{m}{}{X};X)$ is spanned by elements 
	%% ------------ EQUATION ------------ %%
	\begin{equation}
	p_{i,\mbf \alpha}(x) =  (\phi_1^{\alpha_1}(x)\cdots \phi_n^{\alpha_n}(x)) e_i  \qquad (i=1,\dots, n;~ \abs{\alpha}=m).
	\end{equation}
	%% ---------------- END  ---------------- %%
As in the one-dimensional case, we call a map $P : X \to Y$ a polynomial with degree at most $m$ if it can be represented as a sum
	%% ------------ EQUATION ------------ %%
	\begin{equation}
	P(x) = \sum_{j=0}^{m} P_j(x)
	\end{equation}
	%% ---------------- END  ---------------- %%
where $P_{j} \in \mc P_a(\prescript{j}{}{X};Y)$.

The change in degree due to compositions of vector-valued polynomials works as one expects. The proofs of the three following results can be found in the appendix.

\begin{lemma}\label{lem:composition-homogeneous-polynomials}
Let $t, m \geq 2$, $Q_{t} \in \mc P(\prescript{t}{}{X};X)$, and $P = \sum_{k=\ell}^{m} P_{k}(x)$, where $P_k \in \mc P(\prescript{k}{}{X};X)$. Then $Q_{t} \circ P \in \bigoplus_{k=t\ell}^{tm} \mc P(\prescript{k}{}{X};X)$.
\end{lemma}

\begin{corollary}\label{cor:composition-homogeneous-polynomials}
For any $t, m \geq 2$, $\Phi_{t} = I + Q_{t}$, and $P = \sum_{k=\ell}^{m} P_{k}(x)$, where $P_k \in \mc P(\prescript{k}{}{X};X)$ and $Q_{t} \in  \mc P(\prescript{t}{}{X};X)$, then $\Phi_{t} \circ P \in \bigoplus_{k=\ell}^{tm} \mc P(\prescript{k}{}{X};X)$.
\end{corollary}

Every continuous $m$-homogeneous $P$ is differentiable. Denote the differential of $P$ at $a\in X$ by $P'(a)$; the differential satisfies \cite{Mujica:1986uf}
	%% ------------ EQUATION ------------ %%
	\begin{equation}\label{eq:polynomial-differential}
	P'(a)x = m A^s(a^{m-1},x) \qquad (x\in X; m \geq 1)
	\end{equation}
	%% ---------------- END  ---------------- %%
where $A^s$ is a symmetric $m$-linear map defining $P$. Combining \eqref{eq:polynomial-differential} and \eqref{eq:symmetrized-m-linear-form} and listing all the permutations $\sigma$ gives 
	%% ------------ EQUATION ------------ %%
	\begin{equation}\label{eq:polynomial-derivative}
	P'(a)x = A(x,a^{m-1}) + A(a,x,a^{m-2}) + \cdots + A(a^{m-1},x).
	\end{equation}
	%% ---------------- END  ---------------- %%
A function $f : X \to Y$ is a power series around $a\in X$ if
	%% ------------ EQUATION ------------ %%
	\begin{equation*}
	f(x) = \sum_{m=0}^{\infty} P_{m}(x-a) \qquad (P_{m} \in \mc P_a(\prescript{m}{}{X}; Y) ).
	\end{equation*}
	%% ---------------- END  ---------------- %%

\subsection{Principle eigenfunctions of the Koopman operator}\label{subsec:principle-eigenfunctions}
From this point on, we assume that $T: X \to X$ is a diffeomorphism with a power series around its fixed point $x_0 = 0$, and whose differential at 0, $T'(0)$, is diagonalizable with nonzero eigenvalues. In what follows, we consider a \emph{complex} Banach space with dimension $\dim X = n < \infty$ to prove the results in full generality.  Let $C(X) = C_{\C}(X)$ be the space of complex continuous functions on $X$ and define the Koopman operator $U_{T} : C(X) \to C(X)$ associated with $T$ as
	%% ------------ EQUATION ------------ %%
	\begin{equation}
	(U_{T} f)(x) = (f \circ T)(x) \qquad (x\in X).
	\end{equation}
	%% ---------------- END  ---------------- %%
We begin our investigation by defining the principle eigenfunctions of a Koopman operator.

\begin{definition}
Denote the eigenvectors of $T'(0)$ as $\set{ e_1, \dots, e_n } \subset X$, where $T'(0)e_i = \lam_i e_i$, and let $\set{ \phi_1, \dots, \phi_n } \subset X^*$ be the associated coordinate functionals. Each $\phi_i$, for $i=1,\dots, n$, is said to be a principle eigenfunction of $U_{T'(0)}$ associated with $\lam_i$.
\end{definition}

It is easy to show that each coordinate functional $\phi_i$ is a true eigenfunction of $U_{T'(0)}$. Indeed, for all $x \in X$
	%% ------------ EQUATION ------------ %%
	\begin{align*}
	\inner{x}{U_{T'(0)}\phi_i} %&= (U_{T'(0)}\phi_i)(x) = \phi_i(T'(0)x) 
	= \inner{T'(0)x}{\phi_i} 
	= \sum_{j=1}^{n} \lam_j\inner{x}{\phi_{j}}\inner{e_j}{\phi_i} = \lam_i \inner{x}{\phi_i} = \inner{x}{\lam_i \phi_i}.
	\end{align*}
	%% ---------------- END  ---------------- %%	

\begin{definition}
Let $\set{\phi_1, \dots, \phi_n}$ be the principle eigenfunctions of $U_{T'(0)}$. Define $\mc A_{T'(0)} \subset C(X)$ to be the algebra generated by $\set{\phi_1, \dots, \phi_n}$. $\mc A_{T'(0)}$ can be made into a unital algebra by appending the constant functions to it, which are eigenfunctions at 1. We call $\mc A_{T'(0)}$ the principle algebra for $U_{T'(0)}$.
\end{definition}

\begin{proposition}
The elements of the principle algebra $\mc A_{T'(0)}$ are linear combinations of eigenfunctions of $U_{T'(0)}$.
\end{proposition}

\begin{proof}
If $U_{T'(0)} \phi_1 = \lam_1 \phi_1$ and $U_{T'(0)} \phi_2 = \lam_2 \phi_2$, then $U_{T'(0)} (\phi_1\cdot\phi_2) = (\lam_1 \lam_2)(\phi_1\cdot\phi_2)$ where ``$\cdot$'' means pointwise multiplication of the functions \cite{Budisic:2012cf}.
Each member of $\mc A_{T'(0)}$ is of the form
	%% ------------ EQUATION ------------ %%
	\begin{equation*}
	f = \sum_{j=0}^{m} \sum_{\abs{\alpha}=j} c_{j,\alpha} \phi_1^{\alpha_1}\cdots\phi_n^{\alpha_n}, \qquad (m \in \N_0, \alpha \in \N_0^{n}, c_{j,\alpha} \in \C).
	\end{equation*}
	%% ---------------- END  ---------------- %%
Each term $\phi_1^{\alpha_1}\cdots\phi_n^{\alpha_n}$ with $\abs{\alpha}>0$ is an eigenfunction of $U_{T'(0)}$. Additionalyy, since the constant functions are eigenfunctions of $U_{T'(0)}$ at 1, this concludes the proof.
\end{proof}

\begin{corollary}\label{cor:generated-algebra-polynomials}
$\mc A_{T'(0)} = \mc P(X;\C)$, the space of $\C$-valued polynomials on $X$.
\end{corollary}

\begin{proof}
This follows immediately from the canonical representation of polynomials \eqref{eq:canonical-polynomial-map} with $Y$ taken to be $\C$ and that $\set{\phi_1,\dots, \phi_n }$ is a basis for $X^*$.
\end{proof}

These two results answer question \eqref{question:1} from the introduction. 

\begin{remark}
If $\mc A_{T'(0)}$ separates points (which is always true since it is built using a basis $\phi_1,\dots, \phi_n$) for the dual and is closed under complex conjugation (which is always true if $X$ is real and we look at the complexification of $T'(0)$), then, if we restrict the domain of the algebra to be a compact neighborhood $K$ of the origin, the Stone-Weierstrass theorem gives that the uniform closure of the algebra is all the continuous functions on $K$ or the sub-algebra that vanishes at the fixed point (see sec.\ \ref{sec:uniform-closure-of-A}). Now, if $\tau : K \to \tau(L)$ is \emph{any} topological conjugacy from the linear dynamics to the nonlinear dynamics ($T = \tau \circ T'(0) \circ \tau^{-1}$), then the pullback algebra of $\mc A_{T'(0)}$, defined as $\mc A_{T} = \mc A_{T'(0)} \circ \tau^{-1}$, inherits the uniform closure properties of $\mc A_{T'(0)}$. This follows since $K$ and $\tau(K)$ are homeomorphic via $\tau$ (see sec.\ \ref{sec:uniform-closure-of-A}).
\end{remark}

We now define the non-resonance conditions that are fundamental to the existence of a topological conjugacy.

\begin{definition}\label{def:non-resonant-principle-eigenfunctions}
Let $\set{\phi_1, \dots, \phi_n}$ be the principle eigenfunctions of $U_{T'(0)}$ associated with $\set{\lam_1,\dots,\lam_n}$. We say that the principle eigenfunctions are non-resonant up to order $K \in \N$ if for all $\alpha = (\alpha_1,\dots, \alpha_n) \in \N_0^n$, with $2 \leq \abs{\alpha} \leq K$,\footnote{For $\alpha = (\alpha_1,\dots, \alpha_n) \in \N_{0}^{n}$, we define $\abs{\alpha} = \abs{\alpha_1} + \cdots + \abs{\alpha_n}$. } and $j \in \set{1,\dots, n}$, it is true that $\lam_j \neq \lam_{1}^{\alpha_1} \cdots \lam_n^{\alpha_n}$. If the principle eigenfunctions are non-resonant up to order $K$ for all $K \geq 2$, then we just say that they are non-resonant.
\end{definition}

\begin{definition}
We say that the principle algebra $\mc A_{T'(0)}$ is non-resonant if the principle eigenfunctions generating it are non-resonant.
\end{definition}

%%%%%%%%%%%%%%%%%%%%%%%%%%%%%%%%%%%%%%%%%%%%%%%%%%%%%%%%%%%%%%%%%%%%%%%%%%%%%%%%%%%%%%%%%%%%%%%%

\section{Relating the principle algebra to a sub-algebra of observables for the nonlinear dynamics}\label{sec:linear-nonlinear-observables}
Since $T$ is assumed to have a power series around its fixed point, we write it as
	%% ------------ EQUATION ------------ %%
	\begin{equation*}
	T(x) = T'(0)x + P_2(x) + \cdots + P_d(x) + \cdots
	\end{equation*}
	%% ---------------- END  ---------------- %%
where $P_m \in \mc P(\prescript{m}{}{X}; X)$.

\subsection{Normal form computation using the principle algebra}
The construction of the topological conjugacy goes through a sequence of invertible polynomial transformations $x = \Phi_m(y) = (I_{X} + Q_m)(y)$, for $m\geq 2$ and where $I_X$ is the identity operator on $X$, $Q_m \in \mc P(\prescript{m}{}{X}; X)$, is a homogeneneous polynomial, and $x,y \in X$. To generate the topological conjugacy, we inductively eliminate the lowest order polynomial of degree 2 or greater remaining in the series development of $T$ through a sequence of polynomial change of variables $\Phi_{m} : X \to X$.

Let $T_1 = T$ and for $m\geq 2$, define $T_{m} : X \to X$ as
	%% ------------ EQUATION ------------ %%
	\begin{equation}
	T_{m}(x) = (\Phi_{m}^{-1} \circ T_{m-1} \circ \Phi_{m})(x) \qquad (x \in X).
	\end{equation}
	%% ---------------- END  ---------------- %%
where $\Phi_{m}$ is to be chosen later as in the proof of proposition \ref{prop:normal-form-induction} given in appendix \ref{app:normal-form-induction-proof}.

To compute $T_{m}$, we need to compute $\Phi_m^{-1} : X \to X$. The inverse function theorem for Banach spaces guarantees that the inverse exists for all $x$ close enough to the fixed point. Computing the inverse can be accomplished with a standard fixed point algorithm. In what follows, we denote the open unit-ball and the open ball of radius $r > 0$ about the origin as $B = \set{ x \in X \given \norm{x} < 1}$ and $B_r = \set{ x \in X \given \norm{x} < r}$, respectively. The next lemma follows from the inverse function theorem in Banach spaces \cite[][ch. 2]{Aubin:1984tc}. In the construction of the domains of the change of variables $\Phi_{m}$, we can always choose the domain to be contained strictly inside the unit ball.

\begin{lemma}\label{lem:compute-inverse}
For any $m\geq 2$, let $Q_{m} \in \mc P(\prescript{m}{}{X}; X)$ and $\Phi_{m} = I_{X} + Q_m$. There exists $\eps_m > 0$, satisfying $0 < \eps_m < 1$, such that $\Phi_{m}$ has a $C^1$-inverse $\Phi_m^{-1} : \Phi_m(B_{\eps_m}) \to B_{\eps_m}$. For $y\in \Phi_{m}(B_{\eps_m})$, the inverse $x = \Phi_m^{-1}(y)$ can be computed as the limit $x = \lim_{n\to\infty} x_n$ where
	%% ------------ EQUATION ------------ %%
	\begin{equation}\label{eq:inverse-iterative-procedure}
	\begin{aligned}
	x_0 &= 0 \\
	x_{n+1} &= G_y(x_n) = y - Q_{m}(x_n) \qquad (n\geq 0).
	\end{aligned}
	\end{equation}
	%% ---------------- END  ---------------- %%
\end{lemma}

\begin{proof}%

[\emph{Existence}] Note that $\Phi_{m}(0) = 0$ and $\Phi_{m}'(0) = I_X + Q_m'(0)$. Since $m \geq 2$, then $Q_m'(0) = 0$ by \eqref{eq:polynomial-differential}. Therefore $(\Phi_{m}'(0))^{-1} = I_X$. Then \cite[see][ch. 7]{Bollobas:1999ve}, $\Phi_m$ is a homeomorphism of an open neighborhood $U_0$ of $0$ onto an open neighborhood $V_0$ of $\Phi_m(0)$ such that $\Phi_m^{-1} : V_0 \to U_0$ is continuously differentiable with $(\Phi_{m}^{-1})'(y) = (\Phi_{m}'(\Phi_{m}^{-1}(y)))^{-1}$. 

[\emph{Construction}]
Fix $0 < \beta < 1$ and let $A \in \mc L(\prescript{m}{}{X};X)$ be an $m$-linear form defining $Q_{m}$. Let $G_y : X\to X$ be defined as $G_{y}(x) = y - Q_{m}(x)$.  Define $\eps_m > 0$ as a number satisfying $\eps_m < \min\left\{1,  \big(\frac{\beta}{m\norm{A}}\big)^{1/(m-1)} \right\}$ and $B_{\eps_m} \subset U_0$. Then, by \eqref{eq:polynomial-derivative}, for any $a \in B_{\eps_m}$
	%% ------------ EQUATION ------------ %%
	\begin{equation*}
	\norm{G_{y}'(a)} = \norm{Q_{m}'(a)} \leq m \norm{A} \norm{a}^{m-1} \leq \beta < 1.
	\end{equation*}
	%% ---------------- END  ---------------- %%
Since $B_{\eps_m}$ is a convex open set, then $\norm{G_{y}(x) - G_{y}(z)} \leq \beta \norm{x - z}$ for all $x,z \in B_{\eps_m}$ \cite[][Ch.\ 7, Lemma 5]{Bollobas:1999ve}. By the contraction mapping principle, $G_y$ has a unique fixed point $x^* = \lim_{n\to\infty} G^{n}(0)$. That is, $x^* = G(x^*) = y - Q_{m}(x^*)$ which is equivalent to $y = x^{*} + Q_{m}(x^{*}) = \Phi_{m}(x^*)$. Therefore, $x^{*} = \Phi_{m}^{-1}(y)$.
\end{proof}

\begin{remark}
In the construction, we have chosen $\eps_{m}$ such that it is strictly less than 1; i.e., that $B_{\eps_m} \subset B_{1}$.
\end{remark}

\begin{corollary}\label{cor:asymptotic-inverse}
We have that $\Phi_m^{-1}(y) = y - Q_m(y) + e_{m}(y)$, where $e_{m}(y) \in \mc{O}(\norm{y}^{2m-1})$ as $y \to 0$.
\end{corollary}

\begin{proof}
The solution $x = \Phi_{m}^{-1}(y)$ is a fixed point of $G(x) = y - Q_{m}(x)$; 
	%% ------------ EQUATION ------------ %%
	\begin{equation*}
	\Phi^{-1}(y) = x = y - Q_{m}(x) = y - Q_{m}(y - Q_{m}(x)).
	\end{equation*}
	%% ---------------- END  ---------------- %%
Letting $A^s$ be the symmetric $m$-linear form defining $Q_{m}$ and using the Binomial formula
	%% ------------ EQUATION ------------ %%	
	\begin{align*}
	Q_{m}(y - Q_{m}(x)) = A^s( (y - Q_{m}(x))^m ) = Q_{m}(y) + \sum_{j=1}^{m} \binom{m}{j} (-1)^j A^s( y^{m-j}, (Q_{m}(x))^j).
	\end{align*}
	%% ---------------- END  ---------------- %%
Then
	%% ------------ EQUATION ------------ %%
	\begin{align*}
	\norm{Q_{m}(y - Q_{m}(x)) - Q_{m}(y)} &\leq \sum_{j=1}^{m} \binom{m}{j}  \norm{A^s} \norm{y}^{m-j} \norm{Q_{m}(x)}^j \\
	&\leq \sum_{j=1}^{m} \binom{m}{j}  \norm{A^s} \norm{Q_{m}}^j \norm{y}^{m-j} \norm{x}^{mj} \\
	&\leq \sum_{j=1}^{m} \binom{m}{j}  \norm{A^s} \norm{Q_{m}}^j \norm{y}^{m-j} \norm{\Phi^{-1}(y)}^{mj}.
	\end{align*}
	%% ---------------- END  ---------------- %%
Let $U$ be a convex open neighborhood of 0 such that $\overline{U} \subset \Phi_{m}(B_{\eps_m})$. Since $\Phi_{m}^{-1} : \Phi_{m}(B_{\eps_m}) \to B_{\eps_m}$ is $C^1$ on $\Phi_{m}(B_{\eps_m})$, then $\norm{(\Phi_{m}^{-1})'(\cdot)}$ is uniformly bounded by some constant $L > 0 $ on $U$. Then $\Phi_{m}^{-1}$ is Lipschitz with constant $L$ on $U$ \cite[][Ch.\ 7, Lemma 5]{Bollobas:1999ve}. Since $\Phi_{m}^{-1}(0) = 0$, then $\norm{\Phi^{-1}(y)} \leq L \norm{y}$ for all $y \in U$. Therefore,
	%% ------------ EQUATION ------------ %%
	\begin{equation*}
	\norm{Q_{m}(y - Q_{m}(x)) - Q_{m}(y)} \leq \sum_{j=1}^{m} \binom{m}{j}  \norm{A^s} \norm{Q_{m}}^j L^{mj} \norm{y}^{m + (m-1)j} \qquad (y \in U).
	\end{equation*}
	%% ---------------- END  ---------------- %%
The right-hand side is $\mc{O}(\norm{y}^{2m-1})$ as $y \to 0$.
It follows that $Q_{m}(y - Q_{m}(x)) =  Q_{m}(y) + \mc{O}(\norm{y}^{2m-1})$ as $y \to 0$.
\end{proof}

\begin{remark}
The domain of convergence is related to how large an $x$ can be taken such that the differential $Q_{m}'(x)$ can still be uniformly bounded by a constant $\beta < 1$ in some convex neighborhood of 0.
\end{remark}

\begin{lemma}\label{lem:Lie-bracket}
Assume that the principle eigenfunctions of $U_{T'(0)}$ are non-resonant. For each $m\geq 2$, the linear operator $\mc L_{T'(0)}^{(m)} : \mc P(\prescript{m}{}{X}; X) \to \mc P(\prescript{m}{}{X}; X)$ defined as
	%% ------------ EQUATION ------------ %%
	\begin{equation}
	(\mc L_{T'(0)}^{(m)} f)(x) = (U_{T'(0)} f)(x) - T'(0) f(x)
	\end{equation}
	%% ---------------- END  ---------------- %%
has a bounded inverse $(\mc L_{T'(0)}^{(m)})^{-1} : \mc P(\prescript{m}{}{X}; X) \to \mc P(\prescript{m}{}{X}; X)$.
\end{lemma}

\begin{proof}
$\mc P(\prescript{m}{}{X}; X)$ is finite-dimensional with basis elements of the form
	%% ------------ EQUATION ------------ %%
	\begin{equation*}
	p_{j,\mbf \alpha}(x) = e_{j} \prod_{i=1}^{n} \phi_{i}^{\alpha_i}(x), \qquad (j=1,\dots, n;~ \abs{\mbf \alpha} = m).
	\end{equation*}
	%% ---------------- END  ---------------- %%
Then 
	%% ------------ EQUATION ------------ %%
	\begin{align*}
	(\mc L_{T'(0)}^{(m)} p_{j,\mbf \alpha})(x) 
	&= U_{T'(0)} \left( e_{j} \prod_{i=1}^{n} \phi_{i}^{\alpha_i}(x) \right) - T'(0) e_{j} \prod_{i=1}^{n} \phi_{i}^{\alpha_i}(x) \\
	&= e_{j} \prod_{i=1}^{n} \phi_{i}^{\alpha_i}(T'(0)x) - \lam_j e_{j} \prod_{i=1}^{n} \phi_{i}^{\alpha_i}(x) \\
	&= e_{j} \prod_{i=1}^{n} U_{T'(0)}\phi_{i}^{\alpha_i}(x) - \lam_j e_{j} \prod_{i=1}^{n} \phi_{i}^{\alpha_i}(x) \\
	&= e_{j} \prod_{i=1}^{n} \lam_{i}^{\alpha_i}\phi_{i}^{\alpha_i}(x) - \lam_j e_{j} \prod_{i=1}^{n} \phi_{i}^{\alpha_i}(x) \\
	&= \left(\prod_{i=1}^{n} \lam_i^{\alpha_i} - \lam_j \right) p_{j,\mbf\alpha}(x).
	\end{align*}
	%% ---------------- END  ---------------- %%
The non-resonance assumption gives that $\mu_{j,\mbf\alpha} = (\prod_{i=1}^{n} \lam_i^{\alpha_i} - \lam_j) \neq 0$ for all $j=1,\dots, n$ and $\abs{\mbf \alpha} = m \geq 2$. Therefore, $\mc L_{T'(0)}^{(m)}$ is a finite-dimensional, diagonalizable linear operator with non-zero eigenvalues. Therefore, its inverse exists.
\end{proof}

\begin{remark}
There is a similarity transformation $V : \C^{n} \to \mc P(\prescript{m}{}{X}; X)$ such that $\mc L_{T'(0)}^{(m)} = V \Lambda V^{-1}$, where $\Lambda = \diag(\set{\mu_{j,\mbf \alpha}})$, and therefore $(\mc L_{T'(0)}^{(m)})^{-1} = V \Lambda^{-1} V^{-1}$. It follows that $\norm{(\mc L_{T'(0)}^{(m)})^{-1}} \leq \kappa(V)/\mu_{\min}$, where $\mu_{\min} = \min\set{\mu_{j,\mbf\alpha}}$ and $\kappa(V)$ is the condition number of $V$.
\end{remark}

\begin{proposition}\label{prop:normal-form-induction}
Assume that the principle eigenfunctions of $U_{T'(0)}$ are non-resonant. Let $T_{1} = T$ and assume that we have for some $m\geq 1$,  a map $T_{m}(x) = T'(0)x + R_{m+1}(x)$, where $R_{m+1}$ is a power series about 0 and $R_{m+1}(x) \in \mc{O}(\norm{x}^{m+1})$ as $x \to 0$. Then there is a  number $\eps_{m+1}$ satisfying $0 < \eps_{m+1} < 1$ and an invertible polynomial map $\Phi_{m+1} : B_{\eps_{m+1}} \to \Phi_{m+1}(B_{\eps_{m+1}})$ defined as $x = \Phi_{m+1}(z) = (I_X + Q_{m+1})(z)$, where $Q_{m+1} \in \mc P(\prescript{m+1}{}{X};X)$, such that
	%% ------------ EQUATION ------------ %%
	\begin{equation}
	T_{m+1}(z) = (\Phi_{m+1}^{-1} \circ T_{m} \circ \Phi_{m+1})(z) = T'(0)z + R_{m+2}(z),
	\end{equation}
	%% ---------------- END  ---------------- %%
where $R_{m+2}$ is a power series about 0 satisfying $R_{m+2}(z) \in \mc{O}(\norm{z}^{m+2})$ as $z\to 0$.
\end{proposition}

The techniques used in the proof of proposition \ref{prop:normal-form-induction} are the standard ones used in normal form theory \cite[e.g., see][]{Wiggins:2003tz,Katok:1997th}. It consists of introducing a change of variables map $\Phi_{m}$, expanding the resulting substitutions, and then relying on lemma \ref{lem:Lie-bracket} to guarantee the existence of a polynomial that will eliminate the lowest order nonlinear term. For completeness the proof is given in appendix \ref{app:normal-form-induction-proof}

\section{Approximate topological conjugacies and approximate eigenfunctions}\label{sec:approximate-conjugacies-eigenfunctions}
Using the sequence of transformations, $\Phi_{m} : B_{\eps_{m}} \to \Phi_{m}(B_{\eps_{m}})$, from proposition \ref{prop:normal-form-induction}, we wish to define
	%% ------------ EQUATION ------------ %%
	\begin{align}
	\tau_{m} = \Phi_{2} \circ \cdots \circ \Phi_{m},  \qquad (m\geq 2)  .
	\end{align}
	%% ---------------- END  ---------------- %%
We call $\tau_{m}$ the ($m$-th order) approximate conjugacy. Our goal is to define $U_{m} \subset B_{\eps_{m}}$ such that both $\tau_{m}$ and $\tau_{m}^{-1}$ are well-defined. From now on, we denote $B_{\eps_k}$ by $B_{k}$.

Fix $m \geq 2$. Define $\hat U_{m} \subset B_{m}$ by
	%% ------------ EQUATION ------------ %%
	\begin{equation}\label{eq:Um-defn}
	\hat U_{m} := \bigcap_{j=1}^{m-2} \set{ z\in B_{m} \given \Phi_{m-j+1} \circ \cdots \circ \Phi_{m}(z) \in B_{m-j} }.
	\end{equation}
	%% ---------------- END  ---------------- %%
Let $\hat V_{m} := \Phi_2 \circ \cdots \circ \Phi_{m}(\hat U_{m})$.
Graphically, \eqref{eq:Um-defn} is equivalent to the diagram given in figure \ref{fig:nested_cd}.

%\begin{tikzcd}
%B_{m} \supset \hat U_{m} \arrow[r, "\Phi_{m}"] 
%\arrow[rr,  bend left=26, "\Phi_{m-1}\circ\Phi_{m}"]
%\arrow[rrrr,  bend left=42, "\Phi_{3}\circ \cdots \circ\Phi_{m}"] 
%\arrow[rrrrr,  bend left=50, "\Phi_{2}\circ \cdots \circ\Phi_{m}"] 
%& B_{m-1} \arrow[r,"\Phi_{m-1}"]
%& B_{m-2} \arrow[r,"\Phi_{m-2}"] 
%& \cdots \arrow[r,"\Phi_{3}"] 
%& B_{2} \arrow[r,"\Phi_{2}"] 
%& \hat V_{m}
%\end{tikzcd}

\begin{figure}[h]
\begin{center}
\includegraphics[width=0.8\textwidth]{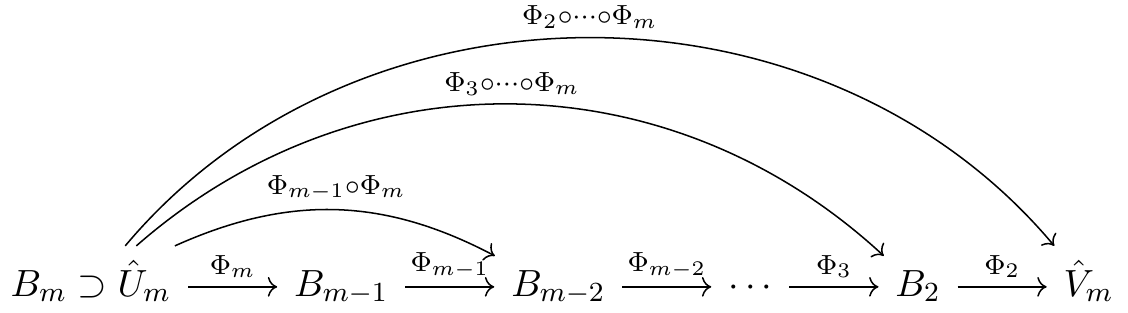}
\caption{Diagram equivalent to \eqref{eq:Um-defn} showing the relationship between the domains $\hat U_{m}$ and the change of variables maps $\Phi_{j}$.}
\label{fig:nested_cd}
\end{center}
\end{figure}

Define $V_{m} \subset \hat V_{m}$ to be a $T$-invariant neighborhood of $0$. Then $\Phi_{m}^{-1} \circ \cdots \circ \Phi_{2}^{-1} : V_{m} \to \hat U_{m}$ is a well-defined map. Finally, define
	%% ------------ EQUATION ------------ %%
	\begin{equation*}
	U_{m} = \Phi_{m}^{-1} \circ \cdots \circ \Phi_{2}^{-1}(V_{m}).
	\end{equation*}
	%% ---------------- END  ---------------- %%
The preceding work shows that $\tau_{m} : U_{m} \to V_{m}$ is well-defined with a $C^1$ inverse, since each $\Phi_{k}$ has a $C^1$ inverse by lemma \ref{lem:compute-inverse}. This implies that 
	%% ------------ EQUATION ------------ %%
	\begin{equation*}
	T_{m} = \tau_{m}^{-1} \circ T \circ \tau_{m} : U_{m} \to U_{m}
	\end{equation*}
	%% ---------------- END  ---------------- %%
is well-defined and invertible (see figure \ref{fig:tau_m_cd}). 

% Commutative diagram
%\begin{tikzcd}
%U_m \arrow[d, "T_{m}"] \arrow[r, "\tau_{m}"] & V_{m} \arrow[d,"T"] \\ 
%U_{m}	& X \arrow[l, "\tau_{m}^{-1}"]
%\end{tikzcd}
\begin{figure}[h]
\begin{center}
\includegraphics[width=0.2\textwidth]{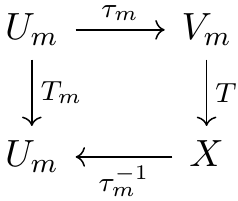}
\caption{Commutative diagram for $T_m$ and $T$.}
\label{fig:tau_m_cd}
\end{center}
\end{figure}

Note that $\Phi_{m+1}(\hat U_{m+1}) \subset \hat U_{m}$, implying that $\hat V_{m+1} \subset \hat V_{m}$. Thus we can choose the $T$-invariant $V_{m+1}$ to satisfy $V_{m+1} \subset V_{m}$ so that $U_{m+1} \subset U_{m}$.
Define 
	%% ------------ EQUATION ------------ %%
	\begin{equation}
	U = \bigcap_{m\geq 2} U_{m}.
	\end{equation}
	%% ---------------- END  ---------------- %%
and
	%% ------------ EQUATION ------------ %%
	\begin{equation}
	V = \bigcap_{m\geq 2} V_{m}.
	\end{equation}
	%% ---------------- END  ---------------- %%
Since $V_{m}$ is a $T$-invariant neighborhood of 0 for all $m\geq 2$, then $V$ is a $T$-invariant set containing 0. If additionally, $V$ has a nonempty interior, then $U$ is a neighborhood of 0. Then $\tau : U \to V$ defined as
	%% ------------ EQUATION ------------ %%
	\begin{equation}\label{eq:limit-conjugacy}
	\tau(z) = \lim_{m\to\infty} \tau_{m}(z)
	\end{equation}
	%% ---------------- END  ---------------- %%
has a well-defined domain and is invertible on $V$. It follows that
	%% ------------ EQUATION ------------ %%
	\begin{equation}
	\norm{\tau(U)} = \sup_{z\in U} \norm{\tau(z)} = \sup_{z\in U} \lim_{m\to\infty} \norm{\tau_m(z)} \leq \norm{\Phi_2(B_{\eps_2)}} < 1.
	\end{equation}
	%% ---------------- END  ---------------- %%

\begin{lemma}
For all $m\geq 2$, $\tau_{m} \in X \otimes \mc A_{T'(0)}$.
\end{lemma}

\begin{proof}
Trivially $\tau_{2} = I_X + Q_{2} \in X \otimes \mc A_{T'(0)}$. We prove the result by induction. For all $m\geq 3$, assume $\tau_{m-1} \in X \otimes \mc A_{T'(0)}$. Now, $\tau_{m} = \tau_{m-1} \circ \Phi_{m}$. Since $X \otimes \mc A_{T'(0)}$ is the space of $X$-valued polynomials on $X$ it is closed under composition. It follows that $\tau_{m} \in X \otimes \mc A_{T'(0)}$ since $Q_{m} \in X \otimes \mc A_{T'(0)}$.
\end{proof}

Given an eigenfunction in $\mc A_{T'(0)}$ of the operator $U_{T'(0)}$, a sequence of approximate eigenfunctions of $U_T$ can be generated by composing the original eigenfunction with the inverse $\tau_{m}^{-1}$ of each approximate conjugacy. One can interpret the following result to say that eigenvalues of $U_{T'(0)}$ are in the pseudo-spectrum of $U_{T}$.

\begin{proposition}\label{prop:approximate-eigenfunction}
Let $\psi \in \mc A_{T'(0)}$ be an eigenfunction of $U_{T'(0)}$ at eigenvalue $\mu \in \C$, $m\geq 2$, and $x = \tau_{m}(z)$ for $z \in U_{m}$. Then 
	%% ------------ EQUATION ------------ %%
	\begin{equation}\label{eq:approximate-eigenfunction}
	U_{T}(\psi \circ \tau_{m}^{-1})(x) = \mu~ (\psi \circ \tau_{m}^{-1})(x) + \mc{O}(\norm{z}^{m+1})
	\end{equation}
	%% ---------------- END  ---------------- %%
as $z \to 0$.
\end{proposition}

\begin{proof}
By definition, $T = \tau_{m} \circ T_{m} \circ \tau_{m}^{-1}$. First, we show, for $x = \tau_{m}(z)$, that $U_{T}(\psi \circ \tau_{m}^{-1})(x) = U_{T_m}\psi(z)$. Indeed,
	%% ------------ EQUATION ------------ %%	
	\begin{align*}
	U_{T}(\psi \circ \tau_{m}^{-1})(x) = (\psi \circ \tau_{m}^{-1}) \circ \tau_{m} \circ T_{m} \circ \tau_{m}^{-1}(x) 
	= \psi \circ T_{m} \circ \tau_{m}^{-1}(x) 
	&= U_{T_m}\psi(\tau_{m}^{-1}(x)) \\
	&= U_{T_m}\psi(z).
	\end{align*}
	%% ---------------- END  ---------------- %%
By proposition \ref{prop:normal-form-induction}, $T_{m}(z) = T'(0)z + R_{m+1}(z)$, where $R_{m+1}(z) \in \mc{O}(\norm{z}^{m+1})$ as $z\to 0$. Then
	%% ------------ EQUATION ------------ %%	
	\begin{align*}
	\norm{U_{T_m}\psi(z) - \mu\, \psi(z)} 
	&= \norm{U_{T_m}\psi(z) - U_{T'(0)}\psi(z)} \\
	&= \norm{\psi\big(T'(0)z + R_{m+1}(z)\big) - \psi(T'(0)z)}.
	\end{align*}
	%% ---------------- END  ---------------- %%
Now, since $\psi$ is in $\mc P(X;\C)$ by corollary \ref{cor:generated-algebra-polynomials}, then $\psi$ has a representation $\psi = c_{\psi} + \sum_{j=1}^{k} A_{j}^{s}$, where $A_{j}^{s}$ is a symmetric $j$-linear form and $c_\psi$ is a constant. Putting $u = T'(0)z$ and $v = R_{m+1}(z)$ and using the Binomial formula (eq. \eqref{eq:binomial-formula}), we get
	%% ------------ EQUATION ------------ %%
	\begin{equation*}
	A_{j}^s((u+v)^m) - A_{j}^{s}(u) = \sum_{i=1}^{j} \binom{j}{i} A^s_{j}(u^{j-i},v^i).
	\end{equation*}
	%% ---------------- END  ---------------- %%
Then, as $z \to 0$,
	%% ------------ EQUATION ------------ %%	
	\begin{align*}
	\norm{A_{j}^s(T'(0)z + R_{m+1}(z)) - A_{j}^s(T'(0)z)} 
	&\leq \sum_{i=1}^{j} \binom{j}{i} \norm{A^s_{j}}\norm{T'(0)z}^{j-i}\norm{R_{m+1}(z)}^{i} \\
	&\leq \sum_{i=1}^{j} \binom{j}{i} \norm{A^s_{j}}\norm{T'(0)}^{j-i}\norm{z}^{j-i}(C\norm{z}^{m+1})^{i} \\
	&\leq C_j \norm{z}^{m+1}\norm{z}^{j-1}.
	\end{align*}
	%% ---------------- END  ---------------- %%
Combining these bounds for each $j=1,\dots,k$ gives
	%% ------------ EQUATION ------------ %%	
	\begin{align*}
	\norm{\psi(T'(0)z + R_{m+1}(z)) - \psi(T'(0)z)} 
	&\leq \sum_{j=1}^{k}\norm{A_{j}^s(T'(0)z + R_{m+1}(z)) - A_{j}^s(T'(0)z)} \\
	&\leq \norm{z}^{m+1} \sum_{j=1}^{k} C_j \norm{z}^{j-1} \\
	& \leq C_{\psi} \norm{z}^{m+1}, \qquad (z\to 0).
	\end{align*}
	%% ---------------- END  ---------------- %%

Therefore, $U_{T_m}\psi(z) - U_{T'(0)}\psi(z) \in \mc{O}(\norm{z}^{m+1})$ as $z\to 0$. It follows that 
	%% ------------ EQUATION ------------ %%	
	\begin{align*}
	U_{T}(\psi\circ \tau_{m}^{-1})(x) 
	&= U_{T_m}\psi(z) = \mu\, \psi(z) +  \mc{O}(\norm{z}^{m+1}) 
	= \mu\, (\psi \circ \tau_{m}^{-1})(x) +  \mc{O}(\norm{z}^{m+1}).
	\end{align*}
	%% ---------------- END  ---------------- %%
\end{proof}

If the limit conjugacy $\tau$ exists so that $T'(0) = \tau^{-1}\circ T \circ \tau$ on a ball a neighborhood $U$, then the approximate eigenfunctions $\psi_m = \psi \circ \tau_{m}^{-1}$ converge pointwise for $x \in V$ to an eigenfunction $\psi \circ \tau^{-1}$ of $U_{T}$.

\begin{corollary}
Assume that the limit conjugacy $\tau : U \to V$ (eq. \eqref{eq:limit-conjugacy}) exists on the neighborhoods $U$ and $V$ of 0. If $\psi \in \mc A$ is an eigenfunction of $U_{T'(0)}$, then $\psi \circ \tau^{-1}$ is an eigenfunction of $U_T$.
\end{corollary}

\begin{proof}
Since $U_{m} \subset B_{\eps_{m}}$ and $\eps_{m} < 1$ (see the remark directly following lemma \ref{lem:compute-inverse}), then $z \in U$ satisfies $\norm{z} < 1$. Letting $m \to \infty$ in \eqref{eq:approximate-eigenfunction} gives the result.
\end{proof}

\section{Real Banach spaces and uniform closures of the principle algebra}\label{sec:uniform-closure-of-A}
Often, we will be working in a real Banach space $X$, and in particular, be interested in observables of a compact neighborhood $K \subset X$ of 0. In this real case, we look at the complexification $X_{\C}$ of $X$ and the extension of $T'(0)$ to $X_{\C}$, which we denote by $T_{\C}'(0) : X_{\C} \to X_{\C}$. If the extended linearization $T_{\C}'(0)$ is diagonalizable, then the uniform closure of the complex algebra $\mc A_{T_{\C}'(0)}$ is either $C_{\C}(K)$, the complex continuous functions on $K$, or the maximal ideal 
	%% ------------ EQUATION ------------ %%
	\begin{equation*}
	I_{0} = \set{ f\in C_{\C}(K) \given f(0) = 0},
	\end{equation*}
	%% ---------------- END  ---------------- %%
depending on whether the constant functions are appended to the generated algebra or not.

Recall that the complexification of a real Banach space $X$ is given by $X_{\C} := \set{ (x,y) \given x,y \in X}$ with the addition operation defined as $(x,y) + (u,v) = (x+u, y+v)$ and scalar multiplication defined as $(\xi + i\zeta)(x,y) = (\xi x - \zeta y, \zeta x + \xi y)$, where $\xi,\zeta \in \R$. Usually, $(x,y) \in X_{\C}$ is written as $x + i y$. The complex eigenvectors $e_{j} = x_{j} + i y_{j} \in X_{\C}$ of $T_{\C}'(0)$ show up in complex conjugate pairs; without loss of generality, we  label the complex conjugate pair of $e_{j}$ as $e_{j+1} := e_{j}^* = x_{j} - iy_{j}$. If $\phi_{j} \in X_{\C}^{*}$ is the dual basis $(\phi_{k}(e_{j}) \equiv \inner{e_{j}}{\phi_{k}} = \delta_{j,k}$), it can be shown that for complex-conjugate pairs $e_{j}$ and $e_{j+1} = e_{j}^{*}$ that
	%% ------------ EQUATION ------------ %%
	\begin{equation}
	\phi_{j+1}(x) = \inner{x}{\phi_{j+1}} = \overline{\inner{x}{\phi_{j}}} = \overline{\phi_{j}(x)},
	\end{equation}
	%% ---------------- END  ---------------- %%
for all $x \in X$, \emph{the real Banach space}. 

We denote by $\mc A_{K}$ the principle algebra generated by the \emph{complex}-valued principle eigenfunctions $\phi_{j}\vert_{K} : K \to \R$, restricted to $K \subset X$, the \emph{real} compact neighborhood of $0$. This algebra is closed under complex conjugation. Indeed, if 
	%% ------------ EQUATION ------------ %%
	\begin{equation*}
	f(x) = \sum_{j=0}^{m} \sum_{\abs{\alpha}=j} c_{j,\alpha} \phi_1(x)^{\alpha_1}\cdots\phi_n(x)^{\alpha_n} \in \mc A_{K}, \qquad (m \in \N_0, \alpha \in \N_0^{n}, c_{j,\alpha} \in \C, x\in K).
	\end{equation*}
	%% ---------------- END  ---------------- %%
then
	%% ------------ EQUATION ------------ %%
	\begin{equation*}
	\overline{f(x)} = \sum_{j=0}^{m} \sum_{\abs{\alpha}=j} \overline{c_{j,\alpha}} \overline{\phi_1(x)}^{\alpha_1}\cdots\overline{\phi_n(x)}^{\alpha_n} \in \mc A_{K}.
	\end{equation*}
	%% ---------------- END  ---------------- %%
This is true since each $\overline{\phi_{j}}$ is either $\phi_{j}$ if it is a real functional or its complex-conjugate $\phi_{j+1}$, otherwise.

\begin{proposition}
Let $X$ be a real Banach space and $T: X\to X$ be a polynomial diffeomorphism with asymptotically stable, hyperbolic fixed point. The uniform closure of $\mc A_{K}$ is 
\begin{compactenum}[(a)]
\item $C_{\C}(K)$ if the constant functions are appended to $\mc A_{K}$, or 
\item the maximal ideal $I_{0} = \set{ f \in C_{\C}(K) \given f(0) = 0}$, otherwise.
\end{compactenum}
\end{proposition}

\begin{proof}
$\mc A_{K}$ is a complex algebra that separates points and is closed under complex conjugation. The Stone-Weirestrass theorem gives the results.
\end{proof}

\begin{corollary}
If the topologically conjugacy $\tau : U \to V$ exists, let $K \subset U$ be a compact neighborhood of 0 and $L = \tau(K) \subset V$. Then $\mc A_{L} = \mc A_{K} \circ \tau^{-1} := \set{f \circ \tau^{-1} \in C_{\C}(L) \given f \in \mc A_{K}}$ is either 
\begin{compactenum}[(a)]
\item $C_{\C}(L)$ if the constant functions are appended to $\mc A_{K}$, or 
\item the maximal ideal $I_{0} = \set{ f \in C_{\C}(L) \given f(0) = 0}$, otherwise.
\end{compactenum}
\end{corollary}

\begin{proof}
The set $L = \tau(K)$ is compact since $\tau$ is continuous and $K$ is compact. It is easy to show that $\mc A_{L}$ is an algebra that is closed under complex conjugation. Since $U$ and $V$ are homeomorphic through $\tau$ and $\mc A_{K}$ separates points, then $\mc A_{L}$ also separates points. The Stone-Weierstrass theorem gives the results.
\end{proof}

%%%%%%%%%%%%%%%%%%%%%%%%%%%%%%%%%%%%%%%%%%%%%%%%%%%%%%%%%%%%%%%%%%%%%%
\section{Conclusions}\label{sec:conclusions}
This paper investigated the connection between the existence of a topological conjugacy between a diffeomorphism $T$ and its linearization $T'(x_0)$ around an asymptotically stable fixed point and the principle eigenfunctions of the Koopman operator $U_{T'(0)}$. The conjugacy $\tau$ from the linear dynamics to the nonlinear dynamics is generated using principle eigenfunctions of $U_{T'(0)}$. Specifically, the principle Koopman eigenfunctions generate an algebra $\mc A_{T'(0)}$ of observables for the linear dynamics from which we can generate a sequence of approximate conjugacies $\set{\tau_{m}}$ satisfying $\tau_m^{-1} \circ T \circ \tau_{m} = T'(x_0) + \mc{O}(\norm{z}^{m+1})$. These approximate conjugacies are in $X \otimes \mc A_{T'(0)}$, the space of $X$-valued polynomials generated by the principle eigenfunctions of $U_{T'(0)}$. As long as the principle eigenfunctions are non-resonant for all orders, then this process can continue in order to get a conjugacy between the linear and nonlinear dynamics $\tau$.

The use of this connection is to generate ``good'' spaces of observables for the nonlinear dynamics. We know by construction that each element in the space of observables $\mc A_{T'(0)}$ for the linear dynamics has an expansion into eigenfunctions of $U_{T'(0)}$. When constructing a space of observables for the nonlinear dynamics, we would also like to have this property. The question is then how to construct such a space and what relation, if any, it has to $\mc A_{T'(0)}$.

If we have built a conjugacy $\tau^{-1}$ from the nonlinear dynamics to the linear dynamics, every element of the space $\mc A_{T'(0)} \circ \tau^{-1} = \set{ g \circ \tau^{-1} \given g \in \mc A_{T'(0)}}$ has an expansion into eigenfunctions of $U_{T}$. The results of this paper build $\tau$, and ultimately $\tau^{-1}$, from a sequence of \emph{polynomial} transformations $\tau_{m}$ generated from the principle Koopman eigenfunctions. $\tau^{-1}$ is an analytic function and for every $g\in \mc A_{T'(0)}$, the observable $g \circ \tau^{-1}$ is an analytic function generated by the principle Koopman eigenfunctions. Thus the Koopman eigenfunctions for the linearized system generate a space of observables for the nonlinear dynamics for which every element has a spectral expansion into eigenfunctions of the Koopman operator associated with the \emph{nonlinear} dynamics. Even if we terminate the construction of the conjugacy at a finite number of steps (giving $\tau_{m}$), we can get an approximate eigenfunction for $U_{T}$ which is within $\mc{O}(\norm{z}^{m+1})$ of a true eigenfunction of $U_{T}$.

Often we work with real Banach spaces and diffeomorphisms on them. If we again linearize around a fixed point and then look at the complexification, we can generate a complex algebra of functions from the principle eigenfunctions that is closed under complex conjugation and whose domain is a compact neighborhood of the fixed point. Appealing to the Stone-Weierstrass theorem allows us to make statements about the uniform closure of the principle algebra. In particular, if the constant functions have been appended to the algebra, the uniform closure of the algebra is the space of complex continuous function. On the other hand, if we have not appended the constant functions to the principle algebra, then every element of the algebra vanishes at the fixed point, and the uniform closure of the algebra is the maximal ideal of complex continuous functions that vanishes at the fixed point. We can get an algebra of observables for the nonlinear dynamics by pulling back the algebra of observables for the linear dynamics using the topological conjugacy. Since the spaces are homeomorphic via the conjugacy, the algebra of observables for the nonlinear dynamics inherits the uniform closure properties of the principle algebra of observables for the linear dynamics. What these results says is that if we choose a continuous observable for our system, we know that it is arbitrarily close, in the uniform norm, to an observable that has a spectral expansion. From another viewpoint, every continuous observable is a perturbation of an observable which has an expansion into Koopman eigenfunctions.

%%%%%%%%%%%%%%%%%%%%%%%%%%%%%%%%%%%%%%%%%%%%%%%%%%%%%%%%%%%%%%%%%%%%%%

\appendix

\section{Compositions of vector-valued polynomials}\label{app:polynomial-proofs}

\begin{proof}[Proof of lemma \ref{lem:composition-homogeneous-polynomials}]
Let $\check Q$ be the $t$-linear map that gives $Q$. By assumption $P(x) = \sum_{k=\ell}^{m} P_{k}(x)$, where $P_{k} \in  \mc P(\prescript{k}{}{X};X)$. Let $\check P_k \in \mc L(\prescript{k}{}{X};X)$ be the $k$-linear map giving $P_k$. $\check Q$ and $\check P_k$ have representations of the form \eqref{eq:canonical-m-linear-map}. Then
	%% ------------ EQUATION ------------ %%
	\begin{align*}
	Q\circ P(x) = \check Q( (P(x))^t ) 
	&= \sum_{j_1,\dots,j_t = 1}^{n} q_{j_1,\dots, j_t} \phi_{j_1}(P(x)) \cdots \phi_{j_t}(P(x)) \\
	&= \sum_{j_1,\dots,j_t = 1}^{n} q_{j_1,\dots, j_t} \prod_{i=1}^{t} \phi_{j_i}(P(x)).
	\end{align*}
	%% ---------------- END  ---------------- %%
where $q_{j_1,\dots, j_t} \in X$.
For any $i \in \set{1,\dots, t}$, we have $\phi_{j_i}(P(x)) = \sum_{k=\ell}^{m}\phi_{j_i}( P_{k}(x) ) = \sum_{k=\ell}^{m}\phi_{j_i}( \check P_{k}(x^k) )$. Let $\check P_{k}(x_1,\dots, x_k)$ be
	%% ------------ EQUATION ------------ %%
	\begin{equation*}
	\check P_{k}(x_1,\dots, x_k) = \sum_{s_1,\dots, s_k = 1}^{n} p_{s_1,\dots,s_k} \prod_{\alpha=1}^{k} \phi_{s_\alpha}(x_\alpha),
	\end{equation*}
	%% ---------------- END  ---------------- %%
where $p_{s_1,\dots,s_k} \in X$ and $s_\alpha \in \set{1,\dots, n}$.
Then 
	%% ------------ EQUATION ------------ %%
	\begin{equation*}
	w_{j_i,k}(x) := \phi_{j_i}(P_k(x)) = \phi_{j_i}( \check P_{k}(x^k) ) = \sum_{s_1,\dots, s_k = 1}^{n} \phi_{j_i}(p_{s_1,\dots,s_k}) \prod_{\alpha=1}^{k} \phi_{s_\alpha}(x) \in \mc P(\prescript{k}{}{X};\C).
	\end{equation*}
	%% ---------------- END  ---------------- %%
It follows that
	%% ------------ EQUATION ------------ %%	
	\begin{align*}
	\prod_{i=1}^{t} \phi_{j_i}(P(x)) = \prod_{i=1}^{t} (\sum_{k=\ell}^{m}\phi_{j_i}( P_{k}(x) ) ) 
	&= \prod_{i=1}^{t} \sum_{k=\ell}^{m} w_{j_i,k}(x)\\
	&=\sum_{j_1,\dots,j_t = 1}^{n} \sum_{k_1,\dots, k_t=\ell}^{m} \prod_{i=1}^{t} w_{j_i,k_i}(x) \\
	&\in \bigoplus_{i = t\ell}^{tm} \mc P(\prescript{k}{}{X};\C).
	\end{align*}
	%% ---------------- END  ---------------- %%
Finally
	%% ------------ EQUATION ------------ %%	
	\begin{align*}
	Q\circ P(x) = \sum_{j_1,\dots,j_\ell = 1}^{n} q_{j_1,\dots, j_\ell} \prod_{i=1}^{\ell} \phi_{j_i}(P(x)) \in \bigoplus_{k=t\ell}^{tm} \mc P(\prescript{k}{}{X};X).
	\end{align*}
	%% ---------------- END  ---------------- %%
\end{proof}

\section{Proof of proposition \ref{prop:normal-form-induction}}\label{app:normal-form-induction-proof}

\begin{proof}
The proof is split into steps for better readability. Let $x = \Phi_{m+1}(z) = (I_X + Q_{m+1})(z)$ and consider $T_{m+1}(z) = (\Phi_{m+1}^{-1}\circ T_{m} \circ \Phi_{m+1})(z)$.
\begin{enumerate}[(i)]
\item By corollary \ref{cor:asymptotic-inverse},
	%% ------------ EQUATION ------------ %%
	\begin{equation}\label{eq:Tm+1}
	T_{m+1}(z) = T_{m}(\Phi_{m+1}(z)) - Q_{m+1}(T_{m}(\Phi_{m+1}(z))) + \mc{O}(\norm{T_{m}(\Phi_{m+1}(z))}^{2m+1})
	\end{equation}
	%% ---------------- END  ---------------- %%
\item Consider $T_{m}(\Phi_{m+1}(z))$:
	%% ------------ EQUATION ------------ %%	
	\begin{align}
	T_{m}(\Phi_{m+1}(z)) &= T'(0)\Phi_{m+1}(z) + R_{m+1}(\Phi_{m+1}(z)) \nonumber \\
	&= T'(0)z + T'(0)Q_{m+1}(z) + R_{m+1}(z + Q_{m+1}(z)) \nonumber \\
	&= T'(0)z + T'(0)Q_{m+1}(z) + R_{m+1}(z) + \mc{O}(\norm{z}^{2m}) \label{eq:Tm-asymptotic}
	\end{align}
	%% ---------------- END  ---------------- %%
where we have used the Binomial formula to expand and bound $R_{m+1}(\Phi_{m+1}(z))$.
\item Consider $Q_{m+1}(T_{m}(\Phi_{m+1}(z)))$:
	%% ------------ EQUATION ------------ %%	
	\begin{align*}
	Q_{m+1}(T_{m}(\Phi_{m+1}(z))) &= Q_{m+1}(T'(0)\Phi_{m+1}(z) + R_{m+1}(\Phi_{m+1}(z))) \\
	&= Q_{m+1}(T'(0)z + T'(0)Q_{m+1}(z) + R_{m+1}(\Phi_{m+1}(z))) \\
	&= Q_{m+1}(T'(0)z + S_{m+1}(z)),
	\end{align*}
	%% ---------------- END  ---------------- %%
where we have defined $S_{m+1}(z) = T'(0)Q_{m+1}(z) + R_{m+1}(\Phi_{m+1}(z))$. Let $B^s$ be the symmetric $(m+1)$-linear form defining $Q_{m+1}$. Then using the Binomial formula
	%% ------------ EQUATION ------------ %%	
	\begin{align*}
	&\norm{Q_{m+1}(T'(0)z + S_{m+1}(z)) - Q_{m+1}(T'(0)z)} \\
	&\qquad\leq \sum_{j=1}^{m+1} \binom{m+1}{j} \norm{B^{s}} \norm{T'(0)z}^{m+1-j} \norm{S_{m+1}(z)}^j \\
	&\qquad\leq \sum_{j=1}^{m+1} \binom{m+1}{j} \norm{B^{s}} \norm{T'(0)}^{m+1-j} \norm{S_{m+1}(z)}^j \norm{z}^{m+1-j}.
	\end{align*}
	%% ---------------- END  ---------------- %%
Furthermore
	%% ------------ EQUATION ------------ %%	
	\begin{align*}
	\norm{S_{m+1}(z)} &\leq \norm{T'(0)Q_{m+1}(z)} + \norm{R_{m+1}(\Phi_{m+1}(z))} \\
	&\leq \norm{T'(0)}\norm{Q_{m+1}(z)} + \norm{R_{m+1}}\norm{\Phi_{m+1}(z)}^{m+1} \\
	&\leq \norm{T'(0)}\norm{Q_{m+1}}\norm{z}^{m+1} + \norm{R_{m+1}}(\norm{\Phi_{m+1}}\norm{z}^{m+1})^{m+1} \\
	&\leq \norm{z}^{m+1} (C_1 + C_2 \norm{z}^{(m+1)^2 - (m+1)})
	\end{align*}
	%% ---------------- END  ---------------- %%
for appropriate positive constants $C_1,C_2$. Therefore
	%% ------------ EQUATION ------------ %%	
	\begin{align*}
	&\norm{Q_{m+1}(T'(0)z + S_{m+1}(z)) - Q_{m+1}(T'(0)z)} \\
	&\qquad \leq \sum_{j=1}^{m+1} \binom{m+1}{j} C_3 \norm{z}^{m+1-j} \norm{z}^{j(m+1)} (C_1 + C_2 \norm{z}^{(m+1)^2 - (m+1)})^{j} \\
	&\qquad \in \mc{O}(\norm{z}^{2m+1}) \qquad (\text{as } y \to 0).
	\end{align*}
	%% ---------------- END  ---------------- %%
Therefore
	%% ------------ EQUATION ------------ %%
	\begin{equation}\label{eq:Qm-asymptotic}
	Q_{m+1}(T_{m}(\Phi_{m+1}(z))) = Q_{m+1}(T'(0)z) + \mc{O}(\norm{z}^{2m+1})\qquad (\text{as } z \to 0).
	\end{equation}
	%% ---------------- END  ---------------- %%
\item Using \eqref{eq:Tm-asymptotic} and \eqref{eq:Qm-asymptotic} in \eqref{eq:Tm+1} gives
	%% ------------ EQUATION ------------ %%	
	\begin{align*}
	T_{m+1}(z) &= T'(0)z + T'(0)Q_{m+1}(z) + R_{m+1}(z) + \mc{O}(\norm{z}^{2m}) \\
	&\qquad - \left[  Q_{m+1}(T'(0)z) + \mc{O}(\norm{z}^{2m+1}) \right] \\
	&\qquad \mc{O}(\norm{T_{m}(\Phi_{m+1}(z))}^{2m+1}) \\
	&= T'(0)z + \left(R_{m+1}(z) - \mc L_{T'(0)}^{(m+1)} Q_{m+1}(z) \right) + \mc{O}(\norm{z}^{2m})
	\end{align*}
	%% ---------------- END  ---------------- %%
Split $R_{m+1}(z) = \hat R_{m+1}(z) + \mc{O}(\norm{z}^{m+2})$ where $\hat R_{m+1} \in \mc P(\prescript{m+1}{}{X}; X)$. Since $\mc L_{T'(0)}^{(m+1)}$ is invertible, we can choose $Q_{m+1}$ to eliminate $\hat R_{m+1}$. With this choice 
	%% ------------ EQUATION ------------ %%
	\begin{equation}
	T_{m+1}(z) = T'(0)z + \mc{O}(\norm{z}^{m+2}) \qquad (z\to 0).
	\end{equation}
	%% ---------------- END  ---------------- %%
It is obvious that $R_{m+2}(z) = T_{m+1}(z) - T'(0)z$ is a power series as it is a composition of polynomials and power series.
\end{enumerate}

This completes the proof.
\end{proof}

% Bibliography
%\bibliographystyle{unsrtnat}
%\bibliography{ref}

\begin{thebibliography}{8}
\providecommand{\natexlab}[1]{#1}
\providecommand{\url}[1]{\texttt{#1}}
\expandafter\ifx\csname urlstyle\endcsname\relax
  \providecommand{\doi}[1]{doi: #1}\else
  \providecommand{\doi}{doi: \begingroup \urlstyle{rm}\Url}\fi

\bibitem[Arnol'd(1989)]{Arnold:1989wt}
V~I Arnol'd.
\newblock \emph{{Mathematical Methods of Classical Mechanics}}.
\newblock Graduate Texts in Mathematics. Springer-Verlag, second edition, May
  1989.

\bibitem[Wiggins(2003)]{Wiggins:2003tz}
Stephen Wiggins.
\newblock \emph{{Introduction to Nonlinear Dynamical Systems and Chaos}},
  volume~2 of \emph{Texts in Applied Mathematics}.
\newblock Springer, 2 edition, 2003.

\bibitem[Katok and Hasselblatt(1997)]{Katok:1997th}
A~B Katok and Boris Hasselblatt.
\newblock \emph{{Introduction to the Modern Theory of Dynamical Systems}}.
\newblock Encyclopedia of Mathematics and its Applications. Cambridge
  University Press, 1997.

\bibitem[Budisic et~al.(2012)Budisic, Mohr, and Mezi{\'c}]{Budisic:2012cf}
Marko Budisic, Ryan Mohr, and Igor Mezi{\'c}.
\newblock {Applied Koopmanism}.
\newblock \emph{Chaos}, 22\penalty0 (4):\penalty0 047510, 2012.

\bibitem[Ryan(2002)]{Ryan:2002ku}
Raymond~A Ryan.
\newblock \emph{{Introduction to Tensor Products of Banach Spaces}}.
\newblock Springer Monographs in Mathematics. Springer London, London, 2002.

\bibitem[Mujica(1986)]{Mujica:1986uf}
Jorge Mujica.
\newblock \emph{{Complex analysis in Banach spaces}}.
\newblock Elsevier, 1986.

\bibitem[Aubin and Ekeland(1984)]{Aubin:1984tc}
Jean-Pierre Aubin and Ivar Ekeland.
\newblock \emph{{Applied nonlinear analysis}}.
\newblock John Wiley {\&} Sons, 1984.

\bibitem[Bollob{\'a}s(1999)]{Bollobas:1999ve}
B{\'e}la Bollob{\'a}s.
\newblock \emph{{Linear analysis}}.
\newblock Cambridge University Press, Cambridge, second edition, 1999.

\end{thebibliography}

%%%%%%%%%%%%%%%%%%%%%%%%%%%%%%%%%%%%%%%%%%%%%%%%%%%%%%%%%%%%%%%%%%%%%%%%%%%%%
% 	BIBLIOGRAPHY
%%%%%%%%%%%%%%%%%%%%%%%%%%%%%%%%%%%%%%%%%%%%%%%%%%%%%%%%%%%%%%%%%%%%%%%%%%%%%

\end{document}